\documentclass[a4paper,11pt]{amsart}

\usepackage{amsmath,amssymb,amsthm,mathtools,hyperref,geometry,cleveref}
\allowdisplaybreaks

\numberwithin{equation}{section}
\theoremstyle{plain}
\newtheorem{thm}{Theorem}[section]
\newtheorem{prop}[thm]{Proposition}
\newtheorem{lem}[thm]{Lemma}
\newtheorem{coro}[thm]{Corollary}

\newtheorem{conj}[thm]{Conjecture}

\newtheorem{rem}[thm]{Remark}

\newtheorem*{acknow}{Acknowledgments}
  

\def\tri{\triangle}
\def\l{\langle}
\def\r{\rangle}
\def\B{\mathcal{B}}
\def\C{\mathcal{C}}
\def\nn{\nabla}
\def\H{H_{\mathbb{R}}}
\def\A{A_{\mathbb{R}}}
\def\AA{\mathring{A}_{\mathbb{R}}}

\title[On Simon's third gap conjecture for minimal surfaces in spheres]{On Simon's third gap conjecture for minimal surfaces in spheres}

\author[W. R. Ding]{Weiran Ding$^{1}$}
\address{$^{1}$School of Mathematical Sciences, South China Normal University, Guangzhou 510000, P. R. CHINA.}
\email{dingwr0806@m.scnu.edu.cn}

\author[J. Q. Ge]{Jianquan Ge$^{2}$}
\address{$^{2}$School of Mathematical Sciences, Laboratory of Mathematics and Complex Systems, Beijing Normal University, Beijing 100875, P. R. CHINA.}
\email{jqge@bnu.edu.cn}

\author[F. G. Li]{Fagui Li$^{3,*}$}
\address{$^{3,*}$Frontier Interdisciplinary Domain, Beijing Institute of Technology, Zhuhai, Guangdong 519088, P. R. CHINA.}
\email{lifagui@bitzh.edu.cn}

\date{}
\thanks{* the corresponding author.}
\thanks{J. Q. Ge is partially supported by NSFC (No. 12571049) and the Fundamental Research Funds for the Central Universities.}
\thanks{F. G. Li is partially supported by  NSFC (No. 12271040, 12501061) and Research Start up Funding of Beijing Institute of Technology (No. 5640011253301).}
\begin{document}
\maketitle

\begin{abstract}
	In this paper, continuing our previous work, we investigate the third gap problem in the Simon conjecture for closed minimal surfaces in the unit sphere. By developing refined third-order Simons-type integral identities and establishing new lower bounds for higher-order curvature terms, we obtain positive gap results throughout the entire interval $\left[\frac{5}{3},\frac{9}{5}\right]$ for the squared norm   of the second fundamental form, including the endpoint cases. %As a consequence, we provide a sharper description of the third gap phenomenon and make further progress toward the resolution of the Simon conjecture.
	 As an application, we establish a rigidity result for closed self-shrinkers.
\end{abstract}

\section{Introduction}
The study of rigidity and gap phenomena for minimal submanifolds in the unit sphere $\mathbb{S}^N$ has a long history and is closely related to several fundamental problems in differential geometry. In 1967, Calabi \cite{Calabi} studied minimal immersions of $\mathbb{S}^2$ with constant Gaussian curvature $K$ into $\mathbb{S}^N$. These immersions are classified up to a rigid motion with the curvature $K$ corresponding to the following values:
\begin{equation*}
	K=K(s)\coloneqq\frac{2}{s(s+1)},\quad s\in\mathbb{N}.
\end{equation*}
An influential problem in this topic is the Simon conjecture, proposed by U. Simon \cite{Simon,LS03} in 1980, concerning the quantization of the Gaussian curvature of closed minimal surfaces in the unit sphere.
\begin{conj}[Intrinsic version]
	Let $M$ be a closed surface minimally immersed in  $\mathbb{S}^N$ such that the image is not contained in any hyperplane of $\mathbb{R}^{N+1}$. If $K(s+1)\leq K\leq K(s)$ for an $s\in \mathbb{N}$, then either $K=K(s+1)$ or $K=K(s)$, and thus the immersion is one of Calabi's $2$-spheres with the dimension of the ambient space $N=2s+2$ or $N=2s$, respectively.
\end{conj}
For minimal surfaces in $\mathbb{S}^N$, the curvature $K$ and the squared norm $S=|h|^2$ of the second fundamental form $h$ are related as follows:
\begin{equation*}\label{KSeq}
	2K=2-S.
\end{equation*}
It follows that, by setting
\begin{equation*}
	S(s)\coloneqq\frac{2(s-1)(s+2)}{s(s+1)}=2-2K(s),
\end{equation*}
the Simon conjecture above can also be stated as:
\begin{conj}[Extrinsic version]
	Let $M$ be a closed surface minimally immersed in  $\mathbb{S}^N$ such that the image is not contained in any hyperplane of $\mathbb{R}^{N+1}$. If $S(s)\leq S\leq S(s+1)$ for an $s\in\mathbb{N}$, then either $S=S(s)$ or $S=S(s+1)$, and thus the immersion is one of Calabi's $2$-spheres with the dimension of the ambient space $N=2s$ or $N=2s+2$, respectively.
\end{conj}
The Simon conjecture is closely related to another rigidity problem named Lu's conjecture \cite{Lu}: \emph{Let $M^n$ be a closed immersed minimal submanifold of the unit sphere $\mathbb{S}^{n+q}$. Let $\lambda_1\ge\lambda_2\ge\cdots\ge\lambda_q$ be the eigenvalues of the fundamental matrix $A=(a_{\alpha\beta})=(\l S_\alpha,S_\beta\r)$, where $\{S_{\alpha}\}_{\alpha=1}^q$ are shape operators with respect to an orthonormal basis $\{\xi_\alpha\}_{\alpha=1}^q$. If $S+\lambda_2$ is a constant and if $S+\lambda_2>n$, then there is a constant $\epsilon(n,q)>0$ such that $S+\lambda_2>n+\epsilon(n,q)$.} For minimal $2$-spheres, and more generally for minimal surfaces under certain curvature inequalities about the normal scalar curvature, Ding-Ge-Li-Yang \cite{DGLY} gave the partial affirmative answer to the Lu conjecture for arbitrary codimension. The Lu conjecture can be viewed as a high-codimensional generalization of a classical gap theorem established by Peng and Terng \cite{PengT}. Indeed, the case $q=1$ was solved by Peng and Terng \cite{PengT}, which is related to the famous Chern conjecture (cf. \cite{CDK,Magl24,TWY,TY}): \emph{Let $M^n$ be a closed immersed minimal hypersurface of the unit sphere $\mathbb{S}^{n+1}$ with constant scalar curvature $R_M$. Then for each $n$, the set of all possible values for $R_M$ is discrete.}
A complete proof of the Chern conjecture remains elusive, though partial results are known in low dimensions or under additional curvature assumptions. The first gap theorem was established by Simons \cite{Simons}, who showed that if $0\le S\le n$, then either $S=0$ or $S=n$ identically on $M$. More recently, it was shown that for a certain class of austere submanifolds, the assumption $0<S\le n$ forces the submanifold to be a Clifford torus \cite{GeTaoZhou}. The high-codimensional formulation of Chern's conjecture likewise remains an open problem. For 2-dimensional minimal surfaces of constant curvature, Chern's conjecture was established by Calabi \cite{Calabi} for minimal 2-spheres and by Bryant \cite{Bryant} for general surfaces. To the best of our knowledge, for high-codimensional submanifolds of dimension three or higher, both Chern's conjecture and Lu's conjecture remain unresolved regarding the second gap. Subsequent work over several decades has addressed the second gap phenomenon, leading to numerous significant contributions (cf. \cite{Chen25,Ding-Xin,LeiXuXu,PengT,XuXu,Yang94,Yang-Cheng3}). We refer to the surveys \cite{GuXuXu,Yau} for a comprehensive overview and further references.

So far, the Simon conjecture has only been solved in the cases $s=1$ and $s=2$ \cite{BKSS,Simon}. To the best of our knowledge, the third gap problem in the Simon conjecture remains widely open, although there are indeed many partial results (cf. \cite{Bolt88a,Bolt88b,It88,Lizhenqi,Ok87,Sa98}) for the case $s\ge3$. In our previous work \cite{DGL}, we made  progress on the case $s=3$ in full generality. By establishing a sequence of Simons-type integral formulas up to third order, we obtained pinching rigidity results for closed minimal surfaces in spheres without imposing any additional assumptions on the normal bundle. Let $$S_{\min}\coloneqq\inf_{p\in M}S(p)\quad\text{and}\quad S_{\max}\coloneqq\sup_{p\in M}S(p),$$ we proved that the gap phenomenon holds throughout the open interval $\frac{5}{3}<S<\frac{9}{5}$, in the sense that a quantitative lower bound for $S_{\max}-S_{\min}$ can be derived \cite{DGL}:
\begin{thm}[Ding-Ge-Li \cite{DGL}]\label{SCMmain}
	Let $M$ be a closed surface minimally immersed in  $\mathbb{S}^N$.
	\begin{enumerate}
		\item\label{ob1} If $0\leq S\leq \frac{4}{3}$, then $S=0$ or $S=\frac{4}{3}$, and the submanifold is Calabi's $2$-sphere with curvature $K\equiv1$ or $K\equiv\frac{1}{3}$, respectively;
		\item\label{ob2} If $\frac{4}{3}\leq S\leq\frac{5}{3},$ then $S=\frac{4}{3}$ or $S=\frac{5}{3}$, and the submanifold is Calabi's $2$-sphere with curvature $K\equiv\frac{1}{3}$ or $K\equiv\frac{1}{6}$, respectively;
		\item\label{ob3} If $\frac{5}{3}\le S\le \frac{9}{5}$, then $$S_{\max}-S_{\min}\ge\frac{134-114S_{\min}+\sqrt{\mathcal{F}}}{108},$$ where $\mathcal{F}=(134-114S_{\min})^2+864(3S_{\min}-5)(9-5S_{\min})$.
	\end{enumerate}
\end{thm}
However, the resulting gap degenerates at the endpoint values $S=\frac{5}{3}$ and $S=\frac{9}{5}$, and no positive gap can be inferred from those estimates alone. This degeneracy reflects an intrinsic limitation of the inequalities (see Theorem \ref{SCMint} (3)) used there:
\begin{equation*}
	\int_M S(3S-4)(3S-5)(5S-9)=\int_M\left[\frac{3}{2}(11S-21)|\nn S|^2-\frac{5}{4}(\tri S)^2+2\mathcal{C}_2+2\mathcal{C}_3\right],
\end{equation*}
where $\mathcal{C}_2$ and $\mathcal{C}_3$ are terms involving the third-order covariant derivatives of the second fundamental form, both of which are nonnegative.
In the final estimates in \cite{DGL}, the nonnegative higher-order terms $2\C_2+2\C_3$ were neglected, and the control of the Laplacian term $(\tri S)^2$ was not optimal. The purpose of the present paper is to overcome these limitations and to establish gap phenomena at both endpoint values, as well as to improve the gap size inside the open interval. Our approach is based on two new ingredients.
\begin{enumerate}
	\item The first new ingredient is a refined analysis of the third-order Simons-type identity. In the previous work, the nonnegative terms $\C_2+\C_3$, arising from the decomposition of the third covariant derivative of the second fundamental form, were simply dropped in the integral estimates. In this paper, we show that these terms admit a strictly positive lower bound under suitable curvature pinching conditions. More precisely, we prove that (see Lemma \ref{infc2c3})
		\begin{equation*}
			\C_2+\C_3\ge\frac{9}{8}S|\nabla S|^2.
		\end{equation*}
		This observation leads to a sharper inequality and plays a crucial role in preventing the degeneration of the gap at the endpoints.
	\item The second new ingredient is an improved estimate for the integral of $(\tri S)^2$. By introducing two auxiliary parameters (denoted here symbolically by $w$ and $t$), we optimize the balance between gradient and Laplacian terms, which yields a significantly stronger integral inequality than those previously available.
\end{enumerate}
By combining these two improvements, we derive a new Simons-type integral inequality that is strictly stronger than the one obtained in our earlier work (see Theorem \ref{newint}, \eqref{equation P parameters wt} and \eqref{equation 1.78 S 1.8 S}):
\begin{equation*}
	\begin{aligned}
		\int_M S(3S-4)(3S-5)(5S-9)\ge \int_M\left[\frac{3}{4}(25S-42)|\nabla S|^2-\frac{5}{4}(\tri S)^2\right]\ge\int_M\mathsf{P}.
	\end{aligned}
\end{equation*}
where $\mathsf{P}$ is a function of $S$ and the auxiliary parameters $w,t$. As a consequence, we establish the existence of positive gaps at both endpoint values $S=\frac{5}{3}$ and $S=\frac{9}{5}$, and at the same time enlarge the gap size throughout the open interval $\left(\frac{5}{3},\frac{9}{5}\right)$. These results provide a more complete picture of the third gap phenomenon in the Simon conjecture and represent a further step toward its full resolution.
\begin{thm}\label{mainthm}
	Let $M$ be a closed minimal surface immersed in $\mathbb{S}^N$. Then we have the following pinching results (see Theorems \ref{left}, \ref{middle} and \ref{right}).
	\begin{enumerate}
		\item\label{zuo} If $\frac{5}{3}\le S\le 1.7075<\frac{9}{5}$, then $S\equiv\frac{5}{3}$ and the submanifold is Calabi's $2$-sphere with curvature $K\equiv\frac{1}{6}$;
		\item\label{zhong} If $\frac{5}{3}\le S\leq \frac{9}{5}$ and $S\not \equiv\frac{5}{3}$, then
			\begin{equation*}
				S_{\max}-S_{\min} 
				\ge\frac{12S_{\min}\left( 9-5S_{\min}\right) (3S_{\min}-4)}
				{60S_{\min}(3S_{\min}-4)+5\left(\frac{19}{4}S_{\min}-\frac{9}{20}\right)^2};
			\end{equation*}
		\item\label{you} If $\frac{5}{3}<1.7853\le S\le\frac{9}{5}$, then $S\equiv\frac{9}{5}$ and the submanifold is Calabi's $2$-sphere with curvature $K\equiv\frac{1}{10}$.
	\end{enumerate}
\end{thm}
\begin{rem}
	Comparing Theorem \ref{mainthm} \eqref{zhong} with Theorem \ref{SCMmain} \eqref{ob3}, we see that the separation between $S_{\max}$ and $S_{\min}$ obtained here is considerably larger than that given previously, indicating a significantly stronger pinching rigidity in the case $s=3$.
\end{rem}

By Theorem \ref{mainthm}, we have
\begin{thm}\label{corollary oscillation}
	Let $M$ be a closed minimal surface immersed in $\mathbb{S}^N$ with $\frac{5}{3} \le S \le \frac{9}{5}$. If $$S_{\max}-S_{\min}\le\frac{1}{220},$$ then $S\equiv\frac{5}{3}$ or $S\equiv\frac{9}{5}$ and the submanifold is Calabi's $2$-sphere with curvature   $K\equiv\frac{1}{6}$ or $K\equiv\frac{1}{10}$, respectively.
\end{thm}
\begin{rem}
	Theorem \ref{corollary oscillation} indicates that the gap phenomenon can be interpreted as a rigidity under small oscillation of the curvature quantity $S$. Meanwhile, Simon's third gap conjecture is equivalent to $	S_{\max}-S_{\min}\le\frac{2}{15}$ in Theorem \ref{corollary oscillation}.
\end{rem}

\begin{coro}\label{corollary no minimal immersion}
	Let $M$ be a closed minimal surface immersed in $\mathbb{S}^N$. Then there exists no isometric minimal immersion $\varPhi:M\to\mathbb{S}^N$ for any $N$ such that $\frac{5}{3}\le S\le\frac{9}{5}$, $S\not\equiv\frac{5}{3}$ and $$S_{\max}<\frac{108S_{\min}(3S_{\min}-4)+5S_{\min}\left(\frac{19}{4}S_{\min}-\frac{9}{20}\right)^2}{60S_{\min}(3S_{\min}-4)+5\left(\frac{19}{4}S_{\min}-\frac{9}{20}\right)^2}.$$
\end{coro}

In the final section we consider the mean curvature flow for a submanifold $M^n$ in $\mathbb{R}^{N}$. Suppose $F:M\times[0,T)\to\mathbb{R}^{n+p}$ is a one-parameter family of smooth isometric immersions. If the position vector $F$ evolves in the direction of the mean curvature vector $\H$, this yields a solution of mean curvature flow:
\begin{equation*}
	\left\{
	\begin{aligned}
		\frac{\partial}{\partial t}F(x,t)&=\H(x,t),\quad x\in M\\
		F(x,0)&=F_0(x),
	\end{aligned}
	\right.
\end{equation*}
where $\H(x,t)$ is the mean curvature vector of $F_t(M)$, which is defined to be the trace of the second fundamental form $\A$, $F_t(x)=F(x,t)$ and $F_0$ is some given immersion. An immersion $F:M^n\to\mathbb{R}^{n+p}$ is called a self-shrinker if it satisfies
\begin{equation}\label{eq:ss-def}
	\H(x)=-\frac{1}{2}F(x)^{\perp}
\end{equation}
for all $x\in M$. We also use the traceless part of the second fundamental form
\begin{equation*}\label{eq:ss-tracefree}
	\AA=\A-\frac{1}{n}g\otimes\H,
\end{equation*}
where $g$ is the induced metric on $M$. Self-shrinkers play a fundamental role in the study of singularities of the mean curvature flow and have been extensively investigated. We refer to \cite{CaoL13,CaoX24,Cold13,Cold12,Ding14,Huis90,XuXu,Zhaoyuhang selfshrinkers25} for some representative works. In this paper, as an application of Theorem \ref{SCMmain} and Theorem \ref{mainthm}, we establish the following rigidity result for closed self-shrinkers.

\begin{thm}\label{thm:ss-three-gaps}
	Let $F:M^2\to\mathbb{R}^{2+p}$ be a closed self-shrinker of the mean curvature flow. Suppose the mean curvature is nowhere vanishing and the normalized mean curvature vector is parallel in the normal bundle.
	\begin{enumerate}
		\item If $$0\le|\AA|^2\le\frac{1}{3},$$ then $M$ is one of the following:
			\begin{enumerate}
				\item[(1a)] $|\AA|^2\equiv 0$ and $M$ is the round sphere $\mathbb{S}^2(2)\subset\mathbb{R}^3\subset\mathbb{R}^{2+p}$
				\item[(1b)] $|\AA|^2\equiv\frac{1}{3}$ and $M$ is the Veronese surface $\mathbb{S}^2(2\sqrt3)\to\mathbb{S}^4(2)\subset\mathbb{R}^5\subset\mathbb{R}^{2+p}$.
			\end{enumerate}
		\item If $$\frac{1}{3}\le|\AA|^2\le \frac{5}{12},$$ then $M$ is one of the following:
			\begin{enumerate}
				\item[(2a)] $|\AA|^2\equiv \frac{1}{3}$ and $M$ is the Veronese surface $\mathbb{S}^2(2\sqrt3)\to\mathbb{S}^4(2)\subset\mathbb{R}^5\subset\mathbb{R}^{2+p}$;
				\item[(2b)] $|\AA|^2\equiv \frac{5}{12}$ and $M$ is Calabi's $2$-sphere $\mathbb{S}^2(2\sqrt6)\to\mathbb{S}^6(2)\subset\mathbb{R}^7\subset\mathbb{R}^{2+p}$.
			\end{enumerate}	
		\item
			\begin{enumerate}
				\item[(3a)] If $$\frac{5}{12}\le|\AA|^2\le 0.426875,$$ then $|\AA|^2\equiv \frac{5}{12}$ and $M$ is Calabi's $2$-sphere $\mathbb{S}^2(2\sqrt6)\to\mathbb{S}^6(2)\subset\mathbb{R}^7\subset\mathbb{R}^{2+p}$.	
				\item[(3b)] If $$0.446325\le|\AA|^2\le \frac{9}{20},$$ then $|\AA|^2\equiv\frac{9}{20}$ and $M$ is Calabi's $2$-sphere $\mathbb{S}^2(2\sqrt{10})\to\mathbb{S}^8(2)\subset\mathbb{R}^9\subset\mathbb{R}^{2+p}$.
				\item[(3c)] If $$
				\frac{5}{12}\le|\AA|^2 \le\frac{9}{20} 
				\quad \text{and} \quad 
				  |\AA|^2_{\max}\le|\AA|^2_{\min}+\frac{1}{880},$$
				 then either $|\AA|^2\equiv\frac{5}{12}$ and $M$ is Calabi's $2$-sphere $\mathbb{S}^2(2\sqrt6)\to\mathbb{S}^6(2)\subset\mathbb{R}^7\subset\mathbb{R}^{2+p}$, or $|\AA|^2\equiv\frac{9}{20}$ and $M$ is Calabi's $2$-sphere $\mathbb{S}^2(2\sqrt{10})\to\mathbb{S}^8(2)\subset\mathbb{R}^9\subset\mathbb{R}^{2+p}$.
			\end{enumerate}
		\end{enumerate}
\end{thm}
\begin{rem}
	Theorem \ref{thm:ss-three-gaps} $(1)$ and $(2)$ were proved by Cao, Xu and Zhao \cite{CaoX24}. The authors \cite{CaoX24} also raised the question of whether the assumptions that the mean curvature is nowhere vanishing and that the normalized mean curvature vector is parallel can be removed.
\end{rem}

\section{Preliminaries}
Let $M$ be a $2$-dimensional manifold immersed in the unit sphere $\mathbb{S}^N$. We assume the range of the indices as follows:
\begin{align*}
	1\leq i,j,k,&\cdots\leq 2,\\
	3\leq\alpha,\beta,\gamma,&\cdots\leq N,\\
	1\leq A,B,C,&\cdots\leq N.
\end{align*}
Let $(e_A)$ be a local orthonormal frame on $T(\mathbb{S}^N)$ such that, when restricted to $M$, $(e_i)$ and $(e_\alpha)$ lie in the tangent bundle $T(M)$ and normal bundle $T^{\bot}(M)$, respectively. We take $(\omega_A)$ and $(\omega_{AB})$ as the metric 1-form field and connection form field associated with $(e_A)$. Let $S_\alpha=(h_{ij}^\alpha)$, where $\omega_{i\alpha}=h_{ij}^\alpha\omega_j$. Then we have $h_{ij}^\alpha=h_{ji}^\alpha$. In the following, we will use the Einstein summation convention. The second fundamental form of $M$ is defined by $h=h_{ij}^\alpha \omega_i\omega_j e_\alpha.$ The mean curvature normal vector field is defined by $2H=h_{ii}^\alpha e_\alpha$. From now on, we assume that the $2$-dimensional manifold $M$ is minimally immersed in $\mathbb{S}^N$, that is to say, the mean curvature normal vector field of $M$ vanishes identically. Let $p=N-2$ be the codimension. Define column vectors
\begin{equation*}
	\begin{aligned}
		a&=(a^\alpha)\coloneqq(h^\alpha_{11})=(-h^\alpha_{22})\in\mathbb{R}^p,\\
		b&=(b^\alpha)\coloneqq(h^\alpha_{12})=(h^\alpha_{21})\in\mathbb{R}^p,\\
		a_i&=(a_i^\alpha)\coloneqq(h^\alpha_{11i})\in\mathbb{R}^p,\\
		a_{ij}&=(a_{ij}^\alpha)\coloneqq(h^\alpha_{11ij})\in\mathbb{R}^p,\\
		a_{ijk}&=(a_{ijk}^\alpha)\coloneqq(h^\alpha_{11ijk})\in\mathbb{R}^p,
	\end{aligned}
\end{equation*}
where the covariant derivatives $h_{ijk}^\alpha$, $h_{ijkl}^\alpha$, $h_{ijklm}^\alpha$ and $h_{ijklmn}^\alpha$ are defined as follows:
\begin{align*}
	h_{ijk}^\alpha\omega_k&=dh_{ij}^\alpha+h_{mj}^\alpha\omega_{mi}+h_{im}^\alpha\omega_{mj}+h_{ij}^\beta\omega_{\beta\alpha},\\
	h_{ijkl}^\alpha\omega_l&=dh_{ijk}^\alpha+h_{mjk}^\alpha\omega_{mi}+h_{imk}^\alpha\omega_{mj}+h_{ijm}^\alpha\omega_{mk}+h_{ijk}^\beta\omega_{\beta\alpha},\\
	h_{ijklm}^\alpha\omega_m&=dh_{ijkl}^\alpha+h_{njkl}^\alpha\omega_{ni}+h_{inkl}^\alpha\omega_{nj}+h_{ijnl}^\alpha\omega_{nk}+h_{ijkn}^\alpha\omega_{nl}+h_{ijkl}^\beta\omega_{\beta\alpha},\\
	h_{ijklmn}^\alpha\omega_n&=dh_{ijklm}^\alpha+h_{pjklm}^\alpha\omega_{pi}+h_{ipklm}^\alpha\omega_{pj}+h_{ijplm}^\alpha\omega_{pk}+h_{ijkpm}^\alpha\omega_{pl}+h_{ijklp}^\alpha\omega_{pm}+h_{ijklm}^\beta\omega_{\beta\alpha}.
\end{align*}
We also use the following notations: $$A\coloneqq(\l S_{\alpha},S_{\beta}\r)=2aa^{\text{T}}+2bb^{\text{T}},\quad S\coloneqq\text{tr}A=|h|^2,\quad\rho_0^{\perp}\coloneqq\sum_{\alpha,\beta}|[S_{\alpha},S_{\beta}]|^2.$$ 
Let
\begin{equation*}
	\begin{aligned}
		\mathcal{B}_1&\coloneqq|\nn h|^2=\sum_{i,j,k,\alpha}(h_{ijk}^\alpha)^2,\\
		\mathcal{B}_2&\coloneqq|\nn^2h|^2=\sum_{i,j,k,l,\alpha}(h_{ijkl}^\alpha)^2,
	\end{aligned}
\end{equation*}
and
\begin{equation*}
	\mathcal{B}_3\coloneqq|\nn^3h|^2=\sum_{i,j,k,l,m,\alpha}(h_{ijklm}^\alpha)^2
\end{equation*}
be the squared lengths of the first, second and third covariant derivatives of $h$, respectively. The Riemannian curvature tensor, the normal curvature tensor and the first covariant differentials of the normal curvature tensor are given by
\begin{align}
	R_{ijkl}&=\frac{1}{2}(2-S)(\delta_{ik}\delta_{jl}-\delta_{il}\delta_{jk}),\label{cur1}\\
	R_{\alpha\beta kl}&=h_{km}^\alpha h_{ml}^\beta-h_{km}^\beta h_{ml}^\alpha,\label{cur2}\\
	R_{\alpha\beta 12k}&=2(b^\beta a_k^\alpha+a^\alpha h_{12k}^\beta-b^\alpha a_k^\beta-a^\beta h_{12k}^\alpha).\label{cur3}
\end{align}
The Codazzi equation and the Ricci formulas are
\begin{align}
	h_{ijk}^\alpha-h_{ikj}^\alpha&=0,\label{Cod}\\
	h_{ijkl}^\alpha-h_{ijlk}^\alpha&=h_{pj}^\alpha R_{pikl}+h_{ip}^\alpha R_{pjkl}+h_{ij}^\beta R_{\beta\alpha kl},\label{Ric1}\\
	h_{ijklm}^\alpha-h_{ijkml}^\alpha&=h_{pjk}^\alpha R_{pilm}+h_{ipk}^\alpha R_{pjlm}+h_{ijp}^\alpha R_{pklm}+h_{ijk}^\beta R_{\beta\alpha lm},\label{Ric2}\\
	h_{ijklmn}^\alpha-h_{ijklnm}^\alpha&=h_{pjkl}^\alpha R_{pimn}+h_{ipkl}^\alpha R_{pjmn}+h_{ijpl}^\alpha R_{pkmn}+h_{ijkp}^\alpha R_{plmn}+h_{ijkl}^\beta R_{\beta\alpha mn}.\label{Ric3}
\end{align}
The Laplacians $\triangle h_{ij}^\alpha$, $\triangle h_{ijk}^\alpha$ and $\triangle h_{ijkl}^\alpha$ are defined by
\begin{equation*}
	\tri h_{ij}^\alpha=\sum_{k}h_{ijkk}^\alpha,\quad\tri h_{ijk}^\alpha=\sum_{l}h_{ijkll}^\alpha,\quad\tri h_{ijkl}^\alpha=\sum_{m}h_{ijklmm}^\alpha.
\end{equation*}
From \eqref{Cod}, \eqref{Ric1}, \eqref{Ric2} and \eqref{Ric3}, we obtain
\begin{equation*}
	\tri h_{ij}^\alpha =h_{mmij}^\alpha+h_{pi}^\alpha R_{pmjm}+h_{mp}^\alpha R_{pijm}+h_{mi}^\delta R_{\delta\alpha jm},
\end{equation*}
\begin{equation*}
	\begin{aligned}
		\tri h_{ijk}^\alpha &=(\tri h_{ij}^\alpha)_k+2h_{pjm}^\alpha R_{pikm}+2h_{ipm}^\alpha R_{pjkm}+h_{ijp}^\alpha R_{pmkm}+2h_{ijm}^\delta R_{\delta\alpha km}\\
		&\hspace{1.3em}+h_{pj}^\alpha R_{pikmm}+h_{ip}^\alpha R_{pjkmm}+h_{ij}^\delta R_{\delta\alpha kmm},
	\end{aligned}	
\end{equation*}
and
\begin{equation*}
	\begin{aligned}
		\tri h^\alpha_{ijkl}&=(\triangle h^\alpha_{ijk})_l+2h^\alpha_{pjkm}R_{pilm}+2h^\alpha_{ipkm}R_{pjlm}+2h^\alpha_{ijpm}R_{pklm}+h^\alpha_{ijkp}R_{pmlm}\\
			&\hspace{1.3em}+h^\alpha_{pjk}R_{pilmm}+h^\alpha_{ipk}R_{pjlmm}+h^\alpha_{ijp}R_{pklmm}+h^\beta_{ijk}R_{\beta\alpha lmm}+2h^\beta_{ijkm}R_{\beta\alpha lm}.
	\end{aligned}
\end{equation*}
The Simons identity \cite{Simons} for minimal surfaces in the unit sphere is
\begin{equation}\label{Sid}
	\frac{1}{2}\tri S=\B_1+2S-|A|^2-\rho_0^\perp.
\end{equation}
Furthermore, the following result is required.
\begin{thm}[Calabi \cite{Calabi}]\label{thm Calabi}
	Let $M$ be a $2$-sphere with a Riemannian metric of constant curvature $K$, and let $X:M\to r\mathbb{S}^{N}\subset\mathbb{R}^{N+1}~(N\ge2)$ be an isometric, minimal immersion of $M$ into the sphere with radius $r$, such that the image is not contained in any hyperplane of $\mathbb{R}^{N+1}$. Then
	\begin{enumerate}
 		\item The dimension $N$ is even, i.e., $N=2s$;
 		\item The value of $K$ is uniquely determined as $$K=\frac{2}{s(s+1)r^2}\eqqcolon K(s,r);$$
 		\item The immersion $X$ is uniquely determined up to a rigid rotation of $r\mathbb{S}^{N}$, and the $N$ components of the vector $X$ are a suitably normalized basis for the spherical harmonics of order $s$ on $M$.
	\end{enumerate}
\end{thm}
The immersion denoted by $\Psi_{2,s}:S^2(K(s))\to\mathbb{S}^{2s}$ is called Calabi's $2$-sphere, where $K(s)=K(s,1)$, $S^2(K(s))$ is the sphere with curvature $K(s)$, and $s=1,2,\cdots$. As we mentioned before, the Simon conjecture has only been solved in the cases $s=1$ and $s=2$. In the following, we recall some results from our previous paper \cite{DGL} without proof.
\begin{lem}[Ding-Ge-Li \cite{DGL}]\label{lemSCM}
	Let $M$ be a closed minimal surface immersed in $\mathbb{S}^N$ with positive Gaussian curvature. Then we have the following propositions.
	\begin{enumerate}
		\item\label{ab} The relationships between $a$ and $b$:
			\begin{equation*}
				\l a,b\r=0\quad\text{and}\quad|a|^2=|b|^2=\frac{1}{4}S;
			\end{equation*}
		\item\label{lap-ab} Laplacians of $a$ and $b$:
			\begin{equation*}
				\tri a=\frac{1}{2}a(4-3S)\quad\text{and}\quad\tri b=\frac{1}{2}b(4-3S);
			\end{equation*}
		\item\label{lap-ai} Laplacians of $a_1$ and $a_2$:
			\begin{equation*}
				\begin{aligned}
					\tri a_1&=\frac{1}{2}a_1(14-9S)+\frac{7}{4}(-aS_1+bS_2),\\
					\tri a_2&=\frac{1}{2}a_2(14-9S)-\frac{7}{4}(bS_1+aS_2);
				\end{aligned}
			\end{equation*}
		\item\label{shapenormal} The relationships between $S$, $|A|^2$ and $\rho^\perp$:
			\begin{equation*}
				|A|^2=\frac{1}{2}S^2\quad\text{and}\quad\rho_0^\perp=S^2;
			\end{equation*}
		\item\label{a1a2} The relationships between $a_1$ and $a_2$:
			\begin{equation*}
				\l a_1,a_2\r=0\quad\text{and}\quad|a_1|^2=|a_2|^2=\frac{1}{8}\mathcal{B}_1;
			\end{equation*}
		\item\label{inneraai} Inner products of $a,b$ and $a_i$:
			\begin{equation*}
				\begin{aligned}
					\l a,a_1\r&=\l b,a_2\r=\frac{1}{8}S_1,\\
					\l a,a_2\r&=-\l b,a_1\r=\frac{1}{8}S_2;
				\end{aligned}
			\end{equation*}
		\item\label{inneraaij} Inner products of $a,b$ and $a_{ij}$:
			\begin{equation*}
				\begin{aligned}
					\l a,a_{11}\r&=\l b,a_{21}\r=\frac{1}{8}(S_{11}-\mathcal{B}_1),\\
					\l a,a_{22}\r&=-\l b,a_{12}\r=\frac{1}{8}(S_{22}-\mathcal{B}_1),\\
					\l a,a_{12}\r&=\l b,a_{22}\r=\frac{1}{8}S_{12},\\
					\l a,a_{21}\r&=-\l b,a_{11}\r=\frac{1}{8}S_{21};
				\end{aligned}
			\end{equation*}
		\item\label{inneraiaij} Inner products of $a_i$ and $a_{ij}$:
			\begin{equation*}
				\begin{aligned}
					\l a_1,a_{22}\r&=-\l a_2,a_{12}\r=\frac{4-3S}{16}S_1-\frac{1}{16}(\B_1)_1,\\
					\l a_2,a_{11}\r&=-\l a_1,a_{21}\r=\frac{4-3S}{16}S_2-\frac{1}{16}(\B_1)_2,\\
					\l a_1,a_{11}\r&=\l a_2,a_{21}\r=\frac{1}{16}(\mathcal{B}_1)_1,\\
					\l a_2,a_{22}\r&=\l a_1,a_{12}\r=\frac{1}{16}(\mathcal{B}_1)_2;
				\end{aligned}
			\end{equation*}
		\item\label{aijaij} The relationships between $a_{11}$, $a_{12}$, $a_{21}$ and $a_{22}$:
			\begin{equation*}
				\begin{aligned}
					\l a_{11},a_{21}\r&=0\quad\text{and}\quad|a_{11}|^2=|a_{21}|^2=\frac{1}{16}\mathcal{B}_2-\frac{1}{32}(3S-4)(S_{11}-S_{22}),\\
					\l a_{22},a_{12}\r&=0\quad\text{and}\quad|a_{22}|^2=|a_{12}|^2=\frac{1}{16}\mathcal{B}_2+\frac{1}{32}(3S-4)(S_{11}-S_{22}).
				\end{aligned}
			\end{equation*}
		\item\label{lapB1} The Laplacian of $\B_1$:
			\begin{equation*}
				\frac{1}{2}\tri\B_1=\frac{7}{2}\tri S-\frac{9}{8}\tri S^2+\frac{1}{2}|\nn S|^2-\frac{1}{4}S(3S-4)(9S-14)+\B_2.
			\end{equation*}
		\item\label{lapB2} The Laplacian of $\B_2$:
			\begin{equation*}
				\begin{aligned}
					\frac{1}{2}\tri\mathcal{B}_2&=(h^\alpha_{ijkl}\tri h^\alpha_{ijk})_l-(21S^2-64S+49)\mathcal{B}_1+7(1-\frac{S}{2})\mathcal{B}_2+\frac{1}{4}S(3S-4)^2(7S-12)\\
					&\hspace{1.4em}-\frac{7}{2}(7S-8)|\nabla S|^2-\l\nabla\mathcal{B}_1,\nabla S\r+\frac{1}{4}(\tri S)^2-\frac{1}{2}|\text{\rm Hess}~S|^2+\mathcal{B}_3.
				\end{aligned}
			\end{equation*}
	\end{enumerate}
\end{lem}
Using these calculations, the authors give the following formulas \cite{DGL}. The proofs are omitted here as they involve rather lengthy and technical calculations.
\begin{thm}[Ding-Ge-Li \cite{DGL}]\label{SCMint}
	Let $M$ be a closed minimal surface immersed in $\mathbb{S}^N$ with positive Gaussian curvature. Then we have the following formulas.
	\begin{enumerate}
		\item\label{1stint} The first gap integral formula
			\begin{equation*}
				\int_M S(3S-4)=2\int_M\mathcal{B}_1\ge0;
			\end{equation*}
		\item\label{2ndint} The second gap integral formula
			\begin{equation*}
				\begin{aligned}
					\int_M S(3S-4)(3S-5)&=2\int_M\left[\frac{1}{2}|\nn S|^2+\C_1\right]\\
					&=2\int_M\left[\mathcal{B}_2-\frac{1}{4}S(3S-4)^2+\frac{1}{2}|\nabla S|^2\right]\ge0,
				\end{aligned}
			\end{equation*}
			where $\C_1=2|a_{11}-a_{22}|^2+2|a_{12}+a_{21}|^2=\B_2-\frac{1}{4}S(3S-4)^2;$
		\item\label{3rdint} The third gap integral formula
			\begin{equation*}
				\begin{aligned}
					\int_M S(3S-4)(3S-5)(5S-9)&=\int_M\left[\frac{3}{2}(11S-21)|\nn S|^2-\frac{5}{4}(\tri S)^2+2\mathcal{C}_2+2\mathcal{C}_3\right]\\
					&=2\int_M\left[\mathcal{B}_3-\frac{1}{8}S(3S-4)(45S^2-144S+116)\right.\\
					&\hspace{4.5em}\left.+\frac{1}{8}(65S-166)|\nn S|^2-\frac{5}{8}(\tri S)^2\right],
				\end{aligned}
			\end{equation*}
			where $\mathcal{C}_2=2|a_{111}-a_{122}|^2+2|a_{211}-a_{222}|^2$, $\mathcal{C}_3=2|a_{112}+a_{121}|^2+2|a_{212}+a_{221}|^2$, and $\mathcal{C}_2+\mathcal{C}_3=\mathcal{B}_3-\frac{1}{4}(45S^2-144S+116)\mathcal{B}_1-\frac{13}{8}(7S-8)|\nabla S|^2$.
	\end{enumerate}
\end{thm}

\section{Proofs of Theorem \ref{mainthm} and Theorem \ref{corollary oscillation}}
First, we need to establish the following lemmas.
\begin{lem}\label{cauchy}
	Let $M$ be a closed Riemannian manifold. Then for all constant $w\in\mathbb{R}$, we have
	\begin{equation*}
		\int_M|\nabla S|^2=-\int_M S\tri S=\int_M(w-S)\tri S\le\sqrt{\int_M(S-w)^2}\sqrt{\int_M(\tri S)^2}.
	\end{equation*}
\end{lem}
\begin{proof}
	The first equality follows by
	\begin{equation*}
		\tri S^2=2S\tri S+2|\nn S|^2,
	\end{equation*}
	the second equality follows since $w\in\mathbb{R}$ is constant, and the last inequality follows by the Cauchy-Schwarz inequality.
\end{proof}
\begin{lem}\label{infc2c3}
	Let $M$ be a closed minimal surface immersed in $\mathbb{S}^N$ with positive Gaussian curvature. Then we have
	\begin{equation*}
		\C_2+\C_3\ge\frac{9}{8}S|\nabla S|^2.
	\end{equation*}
\end{lem}
\begin{proof}
	Using Lemma \ref{lemSCM} \eqref{lap-ab}, we obtain
	\begin{align*}
		(\tri a)_1&=\frac{1}{2}a_1(4-3S)+\frac{1}{2}a(-3S_1)=\frac{4-3S}{2}a_1-\frac{3}{2}aS_1,\\
		(\tri a)_2&=\frac{1}{2}a_2(4-3S)+\frac{1}{2}a(-3S_2)=\frac{4-3S}{2}a_2-\frac{3}{2}aS_2.
	\end{align*}
	Using \eqref{cur1}, \eqref{cur2}, \eqref{cur3} and \eqref{Ric1}, we obtain
	\begin{align*}
		(a_{21}-a_{12})_1&=\frac{4-3S}{2}a_2-\frac{3}{2}bS_1,\\
		(a_{21}-a_{12})_2&=-\frac{4-3S}{2}a_1-\frac{3}{2}bS_2.
	\end{align*}
	Define $E_1\coloneqq a_{111}-a_{122}+a_{212}+a_{221}$ and $E_2\coloneqq a_{112}+a_{121}+a_{222}-a_{211}$. Then we have
	\begin{align*}
		E_1&=(a_{11}+a_{22})_1+(a_{21}-a_{12})_2=-\frac{3}{2}(aS_1+bS_2),\\
		E_2&=(a_{11}+a_{22})_2-(a_{21}-a_{12})_1=\frac{3}{2}(bS_1-aS_2),
	\end{align*}
	which implies that $$|E_1|^2+|E_2|^2=\frac{9}{8}S|\nabla S|^2.$$ On the other hand, by a direct computation, we have
	\begin{align*}
		|E_1|^2+|E_2|^2&=\frac{1}{2}(\C_2+\C_3)+2\l a_{111}-a_{122},a_{212}+a_{221}\r+2\l a_{222}-a_{211},a_{112}+a_{121}\r\\
		&\le\frac{1}{2}(\C_2+\C_3)+2|a_{111}-a_{122}||a_{212}+a_{221}|+2|a_{222}-a_{211}||a_{112}+a_{121}|\\
		&\le\frac{1}{2}(\C_2+\C_3)+\frac{1}{2}(\C_2+\C_3)\\
		&=\C_2+\C_3.
	\end{align*}
	Therefore
	 $$\C_2+\C_3\ge\frac{9}{8}S|\nabla S|^2,$$
	which completes the proof.
\end{proof}
Combining Theorem \ref{SCMint} \eqref{3rdint} and Lemma \ref{infc2c3}, we obtain the following theorem.
\begin{thm}\label{newint}
	Let $M$ be a closed minimal surface immersed in $\mathbb{S}^N$ with positive Gaussian curvature. Then we have
	\begin{equation*}
		\begin{aligned}
			\int_M S(3S-4)(3S-5)(5S-9)\ge\int_M\left[\frac{3}{4}(25S-42)|\nabla S|^2-\frac{5}{4}(\tri S)^2\right].
	\end{aligned}
	\end{equation*}
\end{thm}
By Theorem \ref{newint}, we have the following lemma.
\begin{lem}\label{leftlem1}
	Let $M$ be a closed minimal surface immersed in $\mathbb{S}^N$ with positive Gaussian curvature. Then we have
	\begin{equation*}
		\begin{aligned}
			&\hspace{1.3em}\int_M S(3S-4)(3S-5)(5S-9)\\
			&\ge\int_M\left[-\frac{5(w-S)^2}{16t(1-t)}\left(\frac{2+15t}{2}(w+S)+\frac{36}{5}-2S_{\max}-\frac{126}{5}t\right)^2-\frac{5S(S_{\max}-S)(3S-4)(3S-5)}{2(1-t)}\right],
		\end{aligned}
	\end{equation*}
	for all $0<t<1$ and $w\in\mathbb{R}$.
\end{lem}
\begin{proof}
Combining \eqref{Sid} and Lemma \ref{lemSCM} \eqref{shapenormal}, we obtain
\begin{equation*}
	\frac{1}{2}\tri S=\B_1-\frac{1}{2}S(3S-4).
\end{equation*}
Then we have
\begin{equation*}
	\begin{aligned}
		\int_M(\tri S)^2&=2\int_M\mathcal{B}_1\tri S-\int_M S(3S-4)\tri S\\
		&=2\int_M\mathcal{B}_1\tri S-3\int_M S^2\tri S+4\int_M S\tri S\\
		&=2\int_M\mathcal{B}_1\tri S+\int_M(6S-4)|\nabla S|^2.
	\end{aligned}
\end{equation*}
By the divergence theorem and Lemma \ref{lemSCM} \eqref{lapB1}, we obtain
\begin{equation*}
	\begin{aligned}
		\frac{1}{2}\int_M \mathcal{B}_1\tri S&=\frac{1}{2}\int_M S\tri\mathcal{B}_1\\
		&=\int_M\left(-\frac{7}{2}|\nn S|^2+\frac{11}{4}S|\nn S|^2-\frac{1}{4}S^2(3S-4)(9S-14)+S\mathcal{B}_2\right).
	\end{aligned}
\end{equation*}
Also, by Theorem \ref{SCMint} \eqref{2ndint} we have
\begin{equation}\label{int-C1}
	\begin{aligned}
		\int_M\C_1&=\int_M\left[\frac{1}{2}S(3S-4)(3S-5)-\frac{1}{2}|\nn S|^2\right]\\
		&=\int_M\left[\B_2-\frac{1}{4}S(3S-4)^2\right].
	\end{aligned}
\end{equation}
Hence we obtain
\begin{equation}\label{int-(triS)^2}
	\begin{aligned}
		\int_M(\tri S)^2&=\int_M\left[(6S-4)|\nabla S|^2+(11S-14)|\nabla S|^2-S^2(3S-4)(9S-14)+4S\B_2\right]\\
		&=\int_M\left[(17S-18)|\nabla S|^2-2S^2(3S-4)(3S-5)+4S\C_1\right].
	\end{aligned}
\end{equation}
By \eqref{int-C1} and \eqref{int-(triS)^2} we obtain that
	\begin{equation*}
		\begin{aligned}
			&\hspace{1.3em}\int_M\left[(\tri S)^2+\frac{3(1-t)}{5}(42-25S)|\nn S|^2\right]\\
			&=\int_M\left[\left(17S-18+\frac{3(1-t)}{5}(42-25S)\right)|\nn S|^2+4S\C_1-2S^2(3S-4)(3S-5)\right]\\
			&\le\int_M\left[\left(2S+\frac{36}{5}-\frac{126}{5}t+15tS\right)|\nn S|^2-2S_{\max}|\nn S|^2+2S(S_{\max}-S)(3S-4)(3S-5)\right]\\
			&=\int_M\left[\left(2S-2S_{\max}+\frac{36}{5}-\frac{126}{5}t+15tS\right)|\nn S|^2+2S(S_{\max}-S)(3S-4)(3S-5)\right]\\
			&=\int_M\left[(2+15t)S|\nn S|^2+\left(\frac{36}{5}-2S_{\max}-\frac{126}{5}t\right)|\nn S|^2+2S(S_{\max}-S)(3S-4)(3S-5)\right]\\
			&=\int_M\left[(w-S)\left(\frac{2+15t}{2}(w+S)+\frac{36}{5}-2S_{\max}-\frac{126}{5}t\right)\tri S+2S(S_{\max}-S)(3S-4)(3S-5)\right].
		\end{aligned}
	\end{equation*}
	Thus, by Cauchy-Schwarz's inequality and Young's inequality we have
    \begin{equation*}
		\begin{aligned}
			&\hspace{1.3em}\int_M\left[(\tri S)^2+\frac{3(1-t)}{5}(42-25S)|\nn S|^2\right]\\
			&\le\left[\int_M(w-S)^2\left(\frac{2+15t}{2}(w+S)+\frac{36}{5}-2S_{\max}-\frac{126}{5}t\right)^2\right]^{\frac{1}{2}}\left[\int_M(\tri S)^2\right]^{\frac{1}{2}}\\
			&\hspace{3em}+\int_M2S(S_{\max}-S)(3S-4)(3S-5) \\
			&\le\frac{1}{4t}\int_M(w-S)^2\left(\frac{2+15t}{2}(w+S)+\frac{36}{5}-2S_{\max}-\frac{126}{5}t\right)^2+t\int_M(\tri S)^2\\
			&\hspace{3em}+\int_M S(S_{\max}-S)(3S-4)(3S-5),
		\end{aligned}
	\end{equation*}
	which yields that, for $0<t<1$,
	\begin{equation*}
		\begin{aligned}
			\int_M\left[(\tri S)^2+\frac{3}{5}(42-25S)|\nn S|^2\right]&\le\frac{1}{4t(1-t)}\int_M(w-S)^2\left(\frac{2+15t}{2}(w+S)+\frac{36}{5}-2S_{\max}-\frac{126}{5}t\right)^2\\
			&\hspace{3em}+\int_M\frac{2S}{1-t}(S_{\max}-S)(3S-4)(3S-5).
		\end{aligned}
	\end{equation*}
	Combining Theorem \ref{newint}, we obtain that
	\begin{equation}\label{equation P parameters wt}
		\begin{aligned}
			&\hspace{1.3em}\int_M S(3S-4)(3S-5)(5S-9)\\
			&\ge\int_M\left[\frac{3}{4}(25S-42)|\nn S|^2-\frac{5}{4}(\tri S)^2\right]\\
			&\ge\int_M\left[-\frac{5(w-S)^2}{16t(1-t)}\left(\frac{2+15t}{2}(w+S)+\frac{36}{5}-2S_{\max}-\frac{126}{5}t\right)^2-\frac{5S(S_{\max}-S)(3S-4)(3S-5)}{2(1-t)}\right],
		\end{aligned}
	\end{equation}
	which proves the lemma.
\end{proof}
\subsection{The case $\frac{5}{3}\le S\le 1.7075$.}
\begin{thm}\label{left}
	Let $M$ be a closed minimal surface immersed in $\mathbb{S}^N$. If $\frac{5}{3}\le S\le 1.7075<\frac{9}{5}$, then $S\equiv\frac{5}{3}$ and the submanifold is Calabi's $2$-sphere with curvature $K\equiv\frac{1}{6}$.
\end{thm}
\begin{proof}
	Without loss of generality, we assume that $\frac{5}{3}\le S\le 1.7075$ and $S\not\equiv\frac{5}{3}$. By Lemma \ref{leftlem1} we obtain that
	\begin{equation*}
		\begin{aligned}
			&\hspace{1.3em}\int_M S(3S-4)(3S-5)\left(5S-9+\frac{5(S_{\max}-S)}{2(1-t)}\right)\\
			&\ge\int_M-\frac{5(w-S)^2}{16t(1-t)}\left[\frac{2+15t}{2}(w+S)+\frac{36}{5}-2S_{\max}-\frac{126}{5}t\right]^2.
		\end{aligned}
	\end{equation*}
	If $0<t\le\frac{1}{2}$, then we have
	\begin{equation*}
		\begin{aligned}
			5S-9+\frac{5(S_{\max}-S)}{2(1-t)}&=\frac{5S_{\max}}{2(1-t)}+\frac{5S(1-2t)}{2(1-t)}-9\\
			&\le\frac{5S_{\max}}{2(1-t)}+\frac{5S_{\max}(1-2t)}{2(1-t)}-9\\
			&=5S_{\max}-9.
		\end{aligned}
	\end{equation*}
	Hence, by $\frac{5}{3}\le S\le\frac{9}{5}$, we have
	\begin{equation*}
		\int_M S(3S-4)(3S-5)\left[5S_{\max}-9+\frac{5(w-S)^2\left[\frac{2+15t}{2}(w+S)+\frac{36}{5}-2S_{\max}-\frac{126}{5}t\right]^2}{16t(1-t)S(3S-4)(3S-5)}\right]\ge 0.
	\end{equation*}
	For $\frac{5}{3}\le w\le S\le\frac{9}{5}$, due to $S_{\max}>\frac{5}{3}$ and
	\begin{equation*}
		\frac{(w-S)^2}{(3S-4)(3S-5)}=\frac{1}{9}\left(\frac{4-3w}{3S-4}+1\right)\left(\frac{5-3w}{3S-5}+1\right)\le\frac{(w-S_{\max})^2}{(3S_{\max}-4)(3S_{\max}-5)},
	\end{equation*}
	we derive that
	\begin{equation}\label{equation mil for t and w}
		\int_M S(3S-4)(3S-5)\left[5S_{\max}-9+\frac{5(w-S_{\max})^2\left[\frac{2+15t}{2}(w+S)+\frac{36}{5}-2S_{\max}-\frac{126}{5}t\right]^2}{16t(1-t)S(3S_{\max}-4)(3S_{\max}-5)}\right]\ge 0.
	\end{equation}
	Choosing $t=\frac{1}{2}$, we have
	\begin{equation*}
		\int_M S(3S-4)(3S-5)\left[5S_{\max}-9+\frac{5(w-S_{\max})^2\left[\frac{19}{4}(w+S)-2S_{\max}-\frac{27}{5}\right]^2}{4S(3S_{\max}-4)(3S_{\max}-5)}\right]\ge 0.
	\end{equation*}
	Let
	\begin{equation*}
		\begin{aligned}
			f=\frac{19}{4}(w+S)-2S_{\max}-\frac{27}{5}.
		\end{aligned}
	\end{equation*}
	For $\frac{5}{3}\le w\le\frac{9}{5}$, a direct calculation gives that $f\ge0$ and $\frac{19}{4}w-2S_{\max}-\frac{27}{5}<0$. Then we have
	\begin{equation*}
		\frac{1}{S}f^2=\left[\frac{19}{4}\sqrt{S}+\frac{1}{\sqrt{S}}\left(\frac{19}{4}w-2S_{\max}-\frac{27}{5}\right)\right]^2\le\frac{\left(\frac{11}{4}S_{\max}+\frac{19}{4}w-\frac{27}{5}\right)^2}{S_{\max}}.
	\end{equation*}
	Hence
	\begin{equation*}
		\int_M S(3S-4)(3S-5)\left[5S_{\max}-9+\frac{5(w-S_{\max})^2\left(\frac{11}{4}S_{\max}+\frac{19}{4}w-\frac{27}{5}\right)^2}{4S_{\max}(3S_{\max}-4)(3S_{\max}-5)}\right]\ge 0.
	\end{equation*}
Since $\frac{5}{3}\le S\le\frac{9}{5}$ and $S\not \equiv\frac{5}{3}$, we have
	\begin{equation*}
		5S_{\max}-9+\frac{5(w-S_{\max})^2\left(\frac{11}{4}S_{\max}+\frac{19}{4}w-\frac{27}{5}\right)^2}{4S_{\max}(3S_{\max}-4)(3S_{\max}-5)}\ge0,
	\end{equation*}
	which yields that
	\begin{equation}\label{middleref}
		S_{\max}(3S_{\max}-4)(3S_{\max}-5)(5S_{\max}-9)+\frac{5}{4}(w-S_{\max})^2\left(\frac{11}{4}S_{\max}+\frac{19}{4}w-\frac{27}{5}\right)^2\ge0.
	\end{equation}
	If $w=\frac{5}{3}$, then
	\begin{equation*}
		S_{\max}(3S_{\max}-4)(3S_{\max}-5)(5S_{\max}-9)+\frac{5}{4}\left(\frac{5}{3}-S_{\max}\right)^2\left(\frac{11}{4}S_{\max}+\frac{19}{4}\cdot\frac{5}{3}-\frac{27}{5}\right)^2\ge0,
	\end{equation*}
	which yields that
	\begin{equation*}
		\Theta_1(S_{\max})\coloneqq	S_{\max}(3S_{\max}-4)(5S_{\max}-9)+\frac{5}{36}(3S_{\max}-5)\left(\frac{11}{4}S_{\max}+\frac{151}{60}\right)^2\ge0.
	\end{equation*}
	By a numerical calculation, the root  of $\Theta_1$ lies in the interval $(1.7075,1.7076)$ since $$\Theta_1(1.7075)<0\quad\text{and}\quad\Theta_1(1.7076)>0.$$ Moreover, the root is unique in the interval $\left[\frac{5}{3},\frac{9}{5}\right]$. This implies that $S_{\max}>1.7075$, which contradicts the assumption. Hence, $S\equiv\frac{5}{3}$ and the submanifold is Calabi's $2$-sphere with curvature $K\equiv\frac{1}{6}$ by Theorem \ref{thm Calabi}.
\end{proof}

\subsection{The case $\frac{5}{3}\le S\le \frac{9}{5}$ and  $S\not\equiv\frac{5}{3}$.}
\begin{thm}\label{middle}
	Let $M$ be a closed minimal surface immersed in $\mathbb{S}^N$. If $\frac{5}{3}\le S\leq \frac{9}{5}$ and $S\not \equiv\frac{5}{3}$, then
	\begin{equation*}
		S_{\max}-S_{\min}\ge\frac{12S_{\min}\left( 9-5S_{\min}\right) (3S_{\min}-4)}{60S_{\min}(3S_{\min}-4)+5\left(\frac{19}{4}S_{\min}-\frac{9}{20}\right)^2}.
	\end{equation*}
\end{thm}
\begin{proof}
	For $w\ge\frac{5}{3}$, by \eqref{middleref} we obtain that
	\begin{equation*}\label{euqation S-1.8}
		S_{\max}(3S_{\max}-4)(3S_{\max}-5)(5S_{\max}-9)+\frac{5}{4}(S_{\max}-w)^2\left(\frac{11}{4}S_{\max}+\frac{19}{4}w-\frac{27}{5}\right)^2\ge0.
	\end{equation*}
	Since $S\not\equiv\frac{5}{3}$, we obtain that $\frac{5}{3}<S_{\max}\le\frac{9}{5}$ and
	\begin{equation*}
		\begin{aligned}
			0&\le w(3w-4)(5S_{\max}-9)+\frac{5(S_{\max}-w)^2}{4(3S_{\max}-5)}\left(\frac{11}{4}S_{\max}+\frac{19}{4}w-\frac{27}{5}\right)^2\\
			&=w(3w-4)(5S_{\max}-9)+\frac{5}{12}(S_{\max}-w)\left(1+\frac{5-3w}{3S_{\max}-5}\right)\left(\frac{11}{4}S_{\max}+\frac{19}{4}w-\frac{27}{5}\right)^2\\
			&\le w(3w-4)(5S_{\max}-9)+\frac{5}{12}(S_{\max}-w)\left(\frac{11}{4}\cdot\frac{9}{5}+\frac{19}{4}w-\frac{27}{5}\right)^2\\
			&=w(3w-4)(5S_{\max}-9)+\frac{5}{12}(S_{\max}-w)\left(\frac{19}{4}w-\frac{9}{20}\right)^2.
		\end{aligned}
	\end{equation*}
	Thus
	\begin{equation*}
		5S_{\max}\left(12w(3w-4)+\left(\frac{19}{4}w-\frac{9}{20}\right)^2\right)\ge108w(3w-4)+5w\left(\frac{19}{4}w-\frac{9}{20}\right)^2.
	\end{equation*}
	Therefore we obtain that
	\begin{equation*}
		S_{\max}\ge\frac{108w(3w-4)+5w\left(\frac{19}{4}w-\frac{9}{20}\right)^2}{60w(3w-4)+5\left(\frac{19}{4}w-\frac{9}{20}\right)^2}>w.
	\end{equation*}
	Then the result follows by choosing $w=S_{\min}$.
\end{proof}
\begin{rem}
	If we perform numerical calculations on formula \eqref{equation mil for t and w}  to find a better parameter $t$, we will obtain a better result than  that in  Theorem \ref{middle}.
\end{rem}

\subsection{The case $1.7853\le S\le\frac{9}{5}$.}
\begin{thm}\label{right}
	Let $M$ be a closed minimal surface immersed in $\mathbb{S}^N$. If $\frac{5}{3}<1.7853\le S\le\frac{9}{5}$, then $S\equiv\frac{9}{5}$  and the submanifold is Calabi's $2$-sphere with curvature $K\equiv\frac{1}{10}$.
\end{thm}
\begin{proof}
	Without loss of generality, we assume that $\frac{5}{3}<1.7853\le S\le\frac{9}{5}$ and   $S\not\equiv\frac{9}{5}$.
	By \eqref{int-C1} and \eqref{int-(triS)^2} we obtain that
	\begin{equation*}
		\begin{aligned}
			&\hspace{1.3em}\int_M\left[(\tri S)^2+\frac{3(1-t)}{5}(42-25S)|\nn S|^2\right]\\
			&=\int_M\left[(17S-18)|\nn S|^2-2S^2(3S-4)(3S-5)+4S\C_1+\frac{3(1-t)}{5}(42-25S)|\nn S|^2\right]\\
			&\le\int_M\left[\left(17S-18+\frac{3(1-t)}{5}(42-25S)-2S_{\max}\right)|\nn S|^2+2S(S_{\max}-S)(3S-4)(3S-5)\right]\\
			&=\int_M\left[\left(2S+\frac{36}{5}-\frac{126}{5}t+15tS-2S_{\max}\right)|\nn S|^2+2S(S_{\max}-S)(3S-4)(3S-5)\right]\\
			&\le\int_M\left[\left(15tS_{\max}+\frac{36}{5}-\frac{126}{5}t\right)|\nn S|^2+2S(S_{\max}-S)(3S-4)(3S-5)\right].
		\end{aligned}
	\end{equation*}
	Using Lemma \ref{cauchy}, we obtain
	\begin{equation*}
		\begin{aligned}
			&\hspace{1.3em}\int_M\left[(\tri S)^2+\frac{3}{5}(42-25S)(1-t)|\nn S|^2\right]\\
			&\le\left(15tS_{\max}+\frac{36}{5}-\frac{126}{5}t\right)\sqrt{\int_M(S-w)^2}\sqrt{\int_M(\tri S)^2}+2S(S_{\max}-S)(3S-4)(3S-5)\\
			&\le\frac{\left(15tS_{\max}+\frac{36}{5}-\frac{126}{5}t\right)^2}{4t}\int_M(S-w)^2+t\int_M(\tri S)^2+\int_M2S(S_{\max}-S)(3S-4)(3S-5).
		\end{aligned}
	\end{equation*}
	Hence we obtain that
	\begin{equation*}
		\begin{aligned}
			&\hspace{1.3em}\int_M\left[(\tri S)^2+\frac{3}{5}(42-25S)|\nn S|^2\right]\\
			&\le\frac{\left(15tS_{\max}+\frac{36}{5}-\frac{126}{5}t\right)^2}{4t(1-t)}\int_M(S-w)^2+\frac{2}{1-t}\int_M S(S_{\max}-S)(3S-4)(3S-5),
		\end{aligned}
	\end{equation*}
	which yields that
	\begin{equation}\label{equation 1.78 S 1.8 S}
		\begin{aligned}
			&\hspace{1.3em}\int_M\left[\frac{3}{4}(25S-42)|\nn S|^2-\frac{5}{4}(\tri S)^2\right]\\
			&\ge-\frac{5\left(15tS_{\max}+\frac{36}{5}-\frac{126}{5}t\right)^2}{16t(1-t)}\int_M(S-w)^2-\frac{5}{2(1-t)}\int_M S(S_{\max}-S)(3S-4)(3S-5).
		\end{aligned}
	\end{equation}
	Combining Theorem \ref{newint}, we obtain that
	\begin{equation*}
		\begin{aligned}
			\int_MS(3S-4)(3S-5)\left(5S-9+\frac{5(S_{\max}-S)}{2(1-t)}\right)\ge-\frac{5\left(15tS_{\max}+\frac{36}{5}-\frac{126}{5}t\right)^2}{16t(1-t)}\int_M(S-w)^2.
		\end{aligned}
	\end{equation*}
	Choosing $w=\frac{9}{5}$, we derive that
	\begin{equation*}
		\begin{aligned}
			\int_MS(3S-4)(3S-5)\left(5S-9+\frac{5(S_{\max}-S)}{2(1-t)}\right)\ge-\frac{\left(15tS_{\max}+\frac{36}{5}-\frac{126}{5}t\right)^2}{80t(1-t)}\int_M(5S-9)^2.
		\end{aligned}
	\end{equation*}
	Since $S_{\max}\le\frac{9}{5}$, for $0<t\le\frac{1}{2}$ we obtain that
	\begin{equation*}
		\begin{aligned}
			\int_MS(3S-4)(3S-5)(9-5S)\left(\frac{2t-1}{2-2t}+\frac{\left(15tS_{\max}+\frac{36}{5}-\frac{126}{5}t\right)^2(9-5S)}{80t(1-t)S(3S-4)(3S-5)}\right)\ge0.
		\end{aligned}
	\end{equation*}
By $\frac{5}{3}<S\le\frac{9}{5}$ and   $S\not\equiv\frac{9}{5}$, we have $S_{\min}<\frac{9}{5}$ and
	\begin{equation*}
		\frac{2t-1}{2-2t}S_{\min}(3S_{\min}-4)(3S_{\min}-5)+\frac{\left(27t-\frac{126}{5}t+\frac{36}{5}\right)^2(9-5S_{\min})}{80t(1-t)}\ge0,
	\end{equation*}
	which yields that
	\begin{equation*}
		\Theta_2(S_{\min})\coloneqq40t(2t-1)S_{\min}(3S_{\min}-4)(3S_{\min}-5)+\left(\frac{9}{5}t+\frac{36}{5}\right)^2(9-5S_{\min})\ge0.
	\end{equation*}
	Choosing $t=\frac{1}{4}$, by a numerical calculation we obtain that the root of $\Theta_2$ lies in the interval $(1.7852,1.7853)$ since $$\Theta_2(1.7852)>0\quad\text{and}\quad\Theta_2(1.7853)<0.$$ Moreover, the root is unique in the interval $\left[\frac{5}{3},\frac{9}{5}\right]$. This implies that $S_{\min} < 1.7853$, which contradicts the assumption. Hence, $S\equiv\frac{9}{5}$ and the submanifold is Calabi's $2$-sphere with curvature $K\equiv\frac{1}{10}$ by Theorem \ref{thm Calabi}.
\end{proof}

Now we can give the proofs of Theorem \ref{corollary oscillation} and  Corollary \ref{corollary no minimal immersion}.
\begin{proof}[\textbf{Proof of Theorem \ref{corollary oscillation}}]
	Suppose that $$y(x)=\frac{12x(9-5x)(3x-4)}{60x(3x-4)+5\left(\frac{19}{4}x-\frac{9}{20}\right)^2}.$$ Then we have $$y'(x)=\frac{N'(x)D(x)-N(x)D'(x)}{[D(x)]^2},$$ where $N(x)=-180x^3+564x^2-432x$, $N'(x)=-540x^2+1128x-432$, $D(x)=292.8125x^2-261.375x+1.0125$ and $D'(x)=585.625x-261.375$. 
	%If $\frac{5}{3}\le x\le\frac{9}{5}$, we have $y'(x)<0$ and $y(x)$ is monotonically decreasing. 
	For $x \in [\frac{5}{3}, \frac{9}{5}]$, it follows that $y'(x) < 0$, and thus $y(x)$ is monotonically decreasing.
	This completes the proof by $$y(1.7853)> 0.004565>\frac{1}{220}$$ and Theorem \ref{mainthm}.
\end{proof}

\begin{proof}[\textbf{Proof of Corollary \ref{corollary no minimal immersion}}]
Due to $\frac{5}{3}\le S\le\frac{9}{5}$, $S\not\equiv\frac{5}{3}$ and $$S_{\max}<\frac{108S_{\min}(3S_{\min}-4)+5S_{\min}\left(\frac{19}{4}S_{\min}-\frac{9}{20}\right)^2}{60S_{\min}(3S_{\min}-4)+5\left(\frac{19}{4}S_{\min}-\frac{9}{20}\right)^2},$$
%	The proof is completed by Theorem \ref{mainthm} \eqref{zhong}.
this leads to a contradiction with Theorem \ref{mainthm} \eqref{zhong}. Hence, the existence of such a surface is ruled out.
\end{proof}

\section{Proof of Theorem \ref{thm:ss-three-gaps}}
%Let $\tri$, $\operatorname{div}$  and $d\mu$ be Laplacian, divergence and volume element on $M$, respectively. Colding and Minicozzi \cite{Cold12} introduced a linear operator
%\begin{equation*}\label{eq:ss-drift}
%	\mathcal L=\tri-\frac{1}{2}\langle F,\nabla(\cdot)\rangle=e^{\frac{|F|^2}{4}}\operatorname{div}\left(e^{-\frac{|F|^2}{4}}\nabla(\cdot)\right)
%\end{equation*}
%on self-shrinkers.
%In this  section, we provide the proof  of  Theorem \ref{thm:ss-three-gaps}.
Let  $F: M^n\to\mathbb{R}^{n+p}$ be a self-shrinker of the mean curvature flow.
 By definition, we have
\begin{equation}\label{eq:ss-Acirc}
	|\A|^2=|\AA|^2+\frac{1}{n}|\H|^2.%\text{ %and }|\nabla\A|^2=|\nabla\AA|^2+\frac{1}{2}|\nabla\H|^2.
\end{equation}
The key bridge from self-shrinkers to minimal surfaces in spheres is the following theorem of Smoczyk \cite{Smoc05}.
\begin{thm}[Smoczyk \cite{Smoc05}]\label{thm:smoczyk}
	Let $F:M^n\to\mathbb{R}^{n+p}$ be a closed self-shrinker of the mean curvature flow. Then $M$ is a minimal submanifold of the sphere $\mathbb{S}^{n+p-1}(\sqrt{2n})$ if and only if $\H\neq 0$ and ${\H}/{|\H|}$ is parallel in the normal bundle.
\end{thm}

For 2-dimensional self-shrinkers, we have the following proposition by Theorem \ref{thm:smoczyk}.
%For  surfaces, the relation between the self-shrinker quantity $|\AA|^2$ and the spherical quantity $S$ is especially simple.

\begin{prop}\label{prop:ss-conversion}
	Let $F:M^2\to\mathbb{R}^{2+p}\;(p\ge2)$ be a closed self-shrinker of the mean curvature flow. Suppose the mean curvature is nowhere vanishing and the normalized mean curvature vector is parallel in the normal bundle. Let $A$ and $H$ denote the second fundamental form and mean curvature vector of $M$ when it is viewed as a submanifold of $\mathbb{S}^{1+p}(2)$. Then we have
	\begin{equation*}\label{eq:ss-Atilde-Acirc}
		H\equiv 0\text{ and }|A|^2=|\AA|^2.
	\end{equation*}
\end{prop}

\begin{proof}
	By Theorem \ref{thm:smoczyk}, $M$ is a minimal surface in $\mathbb{S}^{1+p}(2)\subset\mathbb{R}^{2+p}$, which yields that $H\equiv 0$. Let $\A$ and $\H$ be the Euclidean second fundamental form and Euclidean mean curvature vector of the self-shrinker. The Gauss equation gives that
	\begin{equation}\label{eq:ss-gauss-compare}
		\frac{1}{2}+|H|^2-|A|^2=|\H|^2-|\A|^2.
	\end{equation}
	Since $M\subset\mathbb{S}^{1+p}(2)$ and $F^{\perp}=F$ along the sphere, the self-shrinker equation \eqref{eq:ss-def} yields that
	\begin{equation*}
		\H=-\frac{1}{2}F\text{ and }|\H|^2=1.
	\end{equation*}
	Substituting $H\equiv 0$ and $|\H|^2=1$ into \eqref{eq:ss-gauss-compare}, we obtain that
	\begin{equation*}
		|A|^2=|\A|^2-\frac{1}{2}.
	\end{equation*}
	Now \eqref{eq:ss-Acirc} gives that
	\begin{equation*}
		|\AA|^2=|\A|^2-\frac{1}{2}|\H|^2=|\A|^2-\frac{1}{2},
	\end{equation*}
	therefore $|A|^2=|\AA|^2$, which completes the proof.
\end{proof}

We can now give the proof of Theorem \ref{thm:ss-three-gaps}.

\begin{proof}[\textbf{Proof of Theorem \ref{thm:ss-three-gaps}}]
	By Proposition \ref{prop:ss-conversion}, after passing from the closed self-shrinker to the corresponding minimal surface in $\mathbb{S}^{1+p}(2)$ and then rescaling to the unit sphere, the spherical squared norm of the second fundamental form is $$S=4|\AA|^2.$$
	\begin{enumerate}
		\item If $0\le|\AA|^2\le\frac{1}{3}$, then $$0\le S=4|\AA|^2\le \frac43.$$ By Theorem \ref{SCMmain} \eqref{ob1}, either $S\equiv 0$ or $S\equiv\frac{4}{3}$. Hence either $|\AA|^2\equiv 0$ or $|\AA|^2\equiv\frac{1}{3}$. If $S\equiv 0$, then the rescaled minimal surface in the unit sphere is totally geodesic; scaling back by the factor $2$ gives the round sphere $\mathbb{S}^2(2)\subset\mathbb{R}^3\subset\mathbb{R}^{2+p}$. If $S\equiv\frac{4}{3}$, then the rescaled minimal surface is the Veronese surface in $\mathbb{S}^4(1)$; scaling back by the factor $2$ gives the Veronese surface \[\mathbb{S}^2(2\sqrt3)\to \mathbb{S}^4(2)\subset\mathbb{R}^5\subset\mathbb{R}^{2+p}.\] This proves (1).
		\item If $\frac{1}{3}\le|\AA|^2\le\frac{5}{12}$, then $$\frac{4}{3}\le S=4|\AA|^2\le\frac{5}{3}.$$ By Theorem \ref{SCMmain} \eqref{ob2}, either $S\equiv\frac{4}{3}$ or $S\equiv\frac{5}{3}$. Therefore either $|\AA|^2\equiv\frac{1}{3}$ or $|\AA|^2\equiv\frac{5}{12}$. The case $S\equiv\frac{4}{3}$ gives again the Veronese surface. If $S\equiv \frac{5}{3}$, then the rescaled minimal surface is Calabi's $2$-sphere with curvature $\frac{1}{6}$ in the unit sphere; after scaling back by the factor $2$, it becomes Calabi's $2$-sphere $$\mathbb{S}^2(2\sqrt6)\to\mathbb{S}^6(2)\subset\mathbb{R}^7\subset\mathbb{R}^{2+p}.$$ This proves (2).
		\item If $\frac{5}{12}\le|\AA|^2\le 0.426875$, then $$\frac{5}{3}\le S=4|\AA|^2\le 1.7075.$$ Hence Theorem \ref{mainthm} \eqref{zuo} applies and gives $S\equiv\frac53$. Therefore $|\AA|^2\equiv\frac{5}{12}$, and the same scaling argument as above yields Calabi's $2$-sphere $$\mathbb{S}^2(2\sqrt6)\to \mathbb{S}^6(2)\subset\mathbb{R}^7\subset\mathbb{R}^{2+p},$$ which proves (3a). If $0.446325\le|\AA|^2\le\frac{9}{20}$, then $$1.7853\le S=4|\AA|^2\le\frac{9}{5}.$$ Therefore Theorem \ref{mainthm} \eqref{you} applies and gives $S\equiv\frac{9}{5}$. Hence $|\AA|^2\equiv\frac{9}{20}$, and after scaling back by the factor $2$ the corresponding minimal surface is Calabi's $2$-sphere $$\mathbb{S}^2(2\sqrt{10})\to\mathbb{S}^8(2)\subset\mathbb{R}^9\subset\mathbb{R}^{2+p},$$ which proves (3b). The proof of (3c) is similar.
	\end{enumerate}
	Therefore we complete the proof.
\end{proof}

\begin{acknow}
The authors would like to thank Y. H. Zhao for his valuable and constructive suggestions.
%The author is very grateful to Professor dd for their kindly encouragements and supports. 
\end{acknow}
%\begin{data}
%	No data was used for the research described in the article.
%\end{data}

\end{document}